\documentclass[11pt]{article}

\usepackage[usenames]{color}
\usepackage{hyperref}

\usepackage{mathpazo}
\usepackage{amsmath,amssymb,amsfonts}

\newcommand{\update}{15th December 2009}

\title{\vskip-1.2em Group representations with empty residual spectrum}
\author{Y. Choi}
\date{\update}

\renewcommand{\emph}[1]{\textsl{#1}\/} % use slanted font for emphasis
\newcommand{\st}{\;:\;}
\newcommand{\defeq}{:=}

\newcommand{\dt}[1]{\textcolor{Bittersweet}{\textit{\textbf{#1}}}}

\newenvironment{YCnum}{%
\begin{enumerate}

}{\end{enumerate}\ignorespacesafterend}

\newcommand{\veps}{\varepsilon}

\newcommand{\blob}{\bullet}
\newcommand{\iso}{\cong}

\newcommand{\wLOG}{without loss of gen\-er\-ality}
\newcommand{\cu}[1]{{#1}^{\natural}}

\newcommand{\bt}{\operatorname{\mathbf{t}}}

\newcommand{\abs}[1]{\vert{#1}\vert}

\newcommand{\norm}[1]{\Vert{#1}\Vert}

\newcommand{\Cplx}{\mathbb C}
\newcommand{\Real}{\mathbb R}
\newcommand{\Zahl}{\mathbb Z}

\newcommand{\Gm}{\Gamma}
\newcommand{\lm}{\lambda}

\newcommand{\Sp}{\sigma}

\newcommand{\pnorm}[2]{\norm{#2}_{(#1)}}

\newcommand{\Bdd}{{\mathcal B}}

\newcommand{\id}[1][]{\mathbf{1}_{#1}}
\newcommand{\fA}{{\mathfrak A}}
\newcommand{\cC}{{\mathcal C}}
\newcommand{\cK}{{\mathcal K}}
\newcommand{\cM}{{\mathcal M}}
\newcommand{\fg}[1]{{\mathbf{F}_{#1}}}
\newcommand{\Cst}{\operatorname{C}^*}
\newcommand{\VN}{\operatorname{VN}}
\newcommand{\CV}[1]{\operatorname{CV}\nolimits_{#1}}

\newcommand{\pair}[2]{\langle#1,#2\rangle}

\newcommand{\tfae}{the following are equivalent}

\usepackage{amsthm}
\newcounter{pulse}[section]

\newcommand{\Thead}[1]{\sc#1}  % dummy for own personal preference in preprint
\theoremstyle{plain}
\newtheorem{thm}[pulse]{\Thead{Theorem}}
\newtheorem{prop}[pulse]{\Thead{Proposition}}
\newtheorem{lem}[pulse]{\Thead{Lemma}}
\newtheorem{cor}[pulse]{\Thead{Corollary}}
\theoremstyle{definition}
\newtheorem{defn}[pulse]{\Thead{Definition}}

\theoremstyle{remark}
\newtheorem{rem}[pulse]{\Thead{Remark}}

\newtheorem{eg}[pulse]{\Thead{Example}}
\newtheorem{qn}[pulse]{\Thead{Question}}

\begin{document}
\maketitle

\begin{abstract}
Let $X$ be a Banach space on which a discrete group $\Gamma$ acts by isometries. For certain natural choices of $X$, every element of the group algebra, when regarded as an operator on $X$\/, has empty residual spectrum. We show, for instance, that this occurs if $X$ is $\ell^2(\Gm)$ or the group von Neumann algebra $\VN(\Gm)$\/.
In our app\-roach, we introduce the notion of a \dt{surjunctive pair}, and develop some of the basic properties of this construction.

The cases $X=\ell^p(\Gm)$ for $1\leq p<2$ or $2<p<\infty$ are more difficult. If $\Gm$ is amenable we can obtain partial results, using a majorization result of Herz; an example of Willis shows that some condition on $\Gm$ is necessary.

\end{abstract}

\section{Introduction}
If $\Gm$ is a discrete group, and $X$ a
Banach space
on which $\Gm$ acts by translations, elements of the group algebra $\Cplx\Gm$ define bounded linear operators on~$X$\/. The spectral properties of these operators seem little explored when $\Gm$ is non-abelian, except if further strong conditions such as self-adjointness, or having positive coefficients, are imposed.

In particular, we might consider the \dt{residual spectrum} of such an operator. In many cases the residual spectrum turns out to be empty, motivating the following question.

\paragraph{Question.}
Given $\Gamma$ and $X$ as above: does \emph{every} $a\in\Cplx\Gm$ have empty residual spectrum, when regarded as an operator on~$X$?

\medskip

The answer is known to be yes when $\Gm$ is a \dt{Moore group}, that is, a group all of whose irreducible unitary representations are finite-dimensional. This follows by combining the following theorem of %V. 
Runde, with the basic observation (see e.g.~\cite[Proposition 1.3.3]{LarNeu_LST}) that every decomposable operator on a Banach space has empty residual spectrum.

\begin{thm}[\cite{Run_decMoore}]\label{t:Run_decMoore}
Let $G$ be a locally compact Moore group, $X$ a Banach space, and $\theta: L^1(G)\to \Bdd(X)$ a continuous algebra homomorphism. Then $\theta(f)$ is decomposable for each $f\in L^1(G)$.
\end{thm}

If $X$ is finite-dimensional, the residual spectrum of any operator is empty (this just follows by counting dimension). We therefore restrict attention to the infinite-dimensional setting. Moreover, it follows from a previous observation of the author (\cite[Theorem~1]{YC_CMB}; but see Remark~\ref{rem:conjugate} below) that we again obtain a positive answer when $X=\ell^\infty(\Gm)$\/, for \emph{arbitrary}~$\Gm$\/.
In contrast, for $X=\ell^1(\Gm)$ we shall see that the answer depends on our group~$\Gm$\/.
% (The original motivation for considering $\ell^1(\Gm)$ rather than $\ell^\infty(\Gm)$ lies in a dual formulation, about whether surjectivity of convolution operators from $X^*$ to $X^*$ implies their injectivity; the author thanks V. V. Uspenskii for this question.) 
The case $X=\ell^2(\Gm)$ is an obvious choice to consider next, and as a consequence of our main result (Theorem~\ref{t:headline}) we shall see that again the answer to our question is ``yes'' for this choice of $X$\/, regardless of our group~$\Gm$\/.

We attempt in this article to initiate a systematic approach to such questions, developing some machinery in the more general setting of subalgebras of $\Bdd(X)$.
To save needless repetition: given a Banach space $X$ and a
% \emph{unital} 
subalgebra $A\subseteq \Bdd(X)$ -- which need not be norm-closed nor unital -- we say that the pair $(A,X)$ is a \dt{surjunctive pair} if every $T\in A$ has empty residual spectrum. (The terminology will be explained below, see Remark~\ref{rem:terminology}.) The main result of this article can now be stated as follows.

\begin{thm}\label{t:headline}
Let $\Gm$ be a discrete group, and let $X$ be either the reduced group $\Cst$-algebra $\Cst_r(\Gm)$\/, or the non-commutative $L^p$-space associated to the group von Neumann algebra $\VN(\Gm)$, for some $p\in[1,\infty]$.

Let $\Gm$ act by left translations on $X$ in the natural way, and let $\imath: \Cplx\Gm\to\Bdd(X)$ be the induced algebra homomorphism. Then $(\imath(\Cplx\Gm), X)$ is a surjunctive pair.
\end{thm}

The proof of this will be given in Section~\ref{s:cstar}, and actually yields a stronger result when $X$ is one of the non-commutative $L^p$-spaces: in these cases, one can replace $\imath(\Cplx\Gm)$ with the group von Neumann algebra. We note that the noncommutative $L^2$-space associated to $\VN(\Gm)$ is nothing but $\ell^2(\Gm)$, and the $\Gm$-action referred to above is just the left regular representation of $\Gm$ on $\ell^2(\Gm)$.

\begin{rem}\label{rem:conjugate}
We have been lax in describing ``the'' action of $\ell^1(\Gm)$ on $\ell^p(\Gm)$ that was mentioned earlier. Unless specified, we are referring to the homomorphism $L_\blob: \ell^1(\Gm)\to\Bdd(\ell^p(\Gm))$ that is defined by `left translation',~{\it viz.}
\[ L_a: \ell^p(\Gm)\to \ell^p(\Gm) \quad;\quad (L_a\xi)(h)\defeq \sum_{g\in\Gm} a(g) \xi(g^{-1}h) \qquad(a\in \ell^1(\Gm), h\in G). \]

It will be convenient at one point below to consider another algebra homomorphism $\rho_\blob: \ell^1(\Gm)\to \Bdd(\ell^p(\Gm))$\/, which arises by considering the \emph{adjoint} of the natural \emph{right} action of $\ell^1(\Gm)$ on $\ell^q(\Gm))$ for $p^{-1}+q^{-1}=1$\/. Explicitly,
\[ \rho_a: \ell^p(\Gm)\to \ell^p(\Gm) \quad;\quad (\rho_a\xi)(h)\defeq \sum_{g\in\Gm} a(g) \xi(hg) \qquad(a\in \ell^1(\Gm), h\in G). \]
% added Jan 2010
($L_\blob$ is the \dt{left regular representation}, and $\rho_\blob$ the \dt{right regular representation}.)
% end edit
Although $L_\blob$ and $\rho_\blob$ do not coincide, in general, they are inter\-twined by the `flip' operator
\[ \bt : \ell^p(\Gm)\to\ell^p(\Gm) \quad;\quad \bt\xi(h) \defeq \xi(h^{-1})\qquad(h\in\Gm); \]
that is, $\bt \rho_a = L_a \bt$ for all $a\in\ell^1(\Gm)$. It follows that $(L_\blob(\ell^1(\Gm)),\ell^p(\Gm))$ is a surjunctive pair if and only if $(\rho_\blob(\ell^1(\Gm)),\ell^p(\Gm))$~is.
\end{rem}

\medskip
Let us review the outline of this paper. We collect some general preliminaries in the next section. Section~\ref{s:cstar} contains the results used to prove Theorem~\ref{t:headline}, as well as some with applications to CCR groups. Finally, we consider the case where $A=X=\ell^1(\Gamma)$\/. Here the question of surjunctivity remains open, although we obtain a partial result in the case where $\Gm$ is amenable. We close with some specific questions which the work here raises.

\begin{rem}
In view of Theorem~\ref{t:Run_decMoore}, It is natural to wonder what can be said for the left-regular representation of $L^1(G)$ on $L^2(G)$ when $G$ is not discrete. We leave this for future work, although -- as noted above -- in Section~\ref{s:cstar} we can apply some of our general results to the case where $G$ is~CCR.
\end{rem}

\hrule{\small
\paragraph{Note added Jan.\ 2010} After the present work had been submitted for publication, I found that recent work of R.~Tessera gives a complete answer to Question~\ref{q:Tessera} below: see arXiv \href{http://arxiv.org/abs/0801.1532}{0801.1532v4}, Corollary~1.9. A note will be added in proof to the published version of this article.
}
\medskip\hrule

\begin{section}{Preliminaries etc.}

\subsection*{Notation and other conventions}

\paragraph{Notation.} If $X$ is a Banach space and $A\subseteq B(X)$ is a subalgebra, we put
\[ \cu{A}\defeq \{ T+\lambda I \st T\in A, \lm\in\Cplx\}\,.\]
Note that if $A$ is closed in $B(X)$ with respect to the norm topology, then so is~$\cu{A}$.

Recall that if $T:X\to X$ is a bounded linear operator on a Banach space $X$\/, then the \dt{residual spectrum} of $T$\/, which we denote by $\Sp_r(T)$\/, is the set of all $\lambda\in\Sp(T)$ such that $T-\lambda I$ is injective with closed range.
The points of $\Sp(T)\setminus\Sp_r(T)$ are `permanently singular', in the following sense: if there exists a Banach space $Y$\/, into which $X$ embeds as a closed subspace, together with $S\in B(Y)$ such that $S(X)\subseteq X$ and $S\vert_X=T$\/, then $\Sp(T)\setminus\Sp_r(T) \subseteq \Sp(S)$\/.
% A~result of C. J. Read~\cite{Read88_ressp} shows that this is sharp: given $T\in B(X)$ there exists a Banach space $Y$\/, containing $X$ as a closed subspace, and $S\in B(Y)$\/, such that $S(X)\subseteq X$ and $\Sp_r(T)\cap \Sp(S)=\emptyset$\/.

The following lemma is mostly just a translation of the definition (together with an application of the open mapping theorem). We state it as a lemma for later reference, and omit the proof which is straightforward. 

\begin{lem}\label{l:basic}
Let $X$ be a Banach space and $A$ a subalgebra of $\Bdd(X)$. Then \tfae:
\begin{YCnum}
\item the pair $(A,X)$ is surjunctive;
\item whenever $T\in \cu{A}$ is such that $T:X\to X$ is injective with closed range, $T$ is automatically surjective;
\item whenever $T\in \cu{A}$ and $T:X\to X$ is injective but non-invertible, there exists a sequence $(y_n)\subset X$\/, such that $\norm{y_n}\geq 1$ for all $n$ while $\norm{Ty_n}\to 0$\/.
\end{YCnum}
\end{lem}

\begin{rem}\label{rem:terminology}
It is condition (ii) of the preceding lemma which motivates our use of the word `surjunctive'. The term is borrowed from a notion from the theory of dynamical systems, introduced by Gottschalk~\cite{Gott_surjunc} in the context of group actions on certain compact spaces. The corresponding notion for Banach spaces seems not to have been systematically considered.
\end{rem}

\subsection*{Various background remarks and results}
In the definition of a surjunctive pair, we did not insist that $A$ is a \emph{closed} subalgebra (this gives us slightly greater flexibility in the statements of our results).
The following lemma shows that this is not an important distinction.

\begin{lem}
Let $(A,X)$ be a surjunctive pair and let $\fA$ be the closure of $A$ in the norm topology of $\Bdd(X)$. Then $(\fA,X)$ is a surjunctive pair.
\end{lem}

\begin{proof}
Let $T\in \cu{\fA}$ be injective with closed range.
By the open mapping theorem, there exists $\delta>0$ such that $\norm{T(x)}\geq \delta\norm{x}$ for all $x\in X$. 
Since $\cu{A}$ is norm-dense in $\cu{\fA}$, there exists $T_1\in \cu{A}$ such that $\norm{T-T_1} \leq \delta/3$. Then since $\norm{T_1(x)}\geq (2\delta/3)\norm{x}$ for all $x\in X$, we see that $T_1$ is injective with closed range; as $(A,X)$ is a surjunctive pair, $T_1$ is therefore invertible in $\Bdd(X)$. Note that $\norm{T_1^{-1}} \leq 3(2\delta)^{-1}$\/, 
which yields
\[ \norm{T_1^{-1}T - I} \leq \norm{T_1^{-1}} \cdot \norm{T-T_1} \leq \frac{3}{2\delta} \cdot \frac{\delta}{3} = \frac{1}{2} \,.\]
It follows that $T_1^{-1}T$ is invertible, whence $T$ itself is invertible, as required.
\end{proof}

\begin{rem}\label{r:RumDoodle}
We shall see in Example~\ref{eg:SOTclosure-not-surjunctive} that in general, ``norm closure'' cannot be replaced by ``closure in the strong operator topology''.
\end{rem}

\begin{lem}\label{l:soi-meme}
Let $X$ be a Banach space, $A\subseteq \Bdd(X)$ a closed subalgebra.
If $(A,\cu{A})$, where we let $A$ act on itself by left multiplication, is surjunctive, then so is $(A,X)$.
\end{lem}

\begin{proof}
Let $T\in\cu{A}\subseteq B(X),$ and suppose that $T:X\to X$ is injective but not surjective. Let $L_T: \cu{A} \to \cu{A}$ denote the operator $S\mapsto TS$\/. Injectivity of $T$ implies that $L_T$ is injective.
Moreover, $L_T$ is not surjective: for if it were, then since $I\in\cu{A}$ there would exist $S\in \cu{A}$ such that $TS=I$\/, and so $T$ would be surjective, contrary to our original assumption.

Since $L_T$ is injective but not surjective, and $(A,\cu{A})$ is a surjunctive pair, $L_T$ must have non-closed range; hence there exists a sequence $(S_n)\subset \cu{A}$ such that $\norm{S_n}=1$ for all $n$ while $\norm{TS_n}\to 0$\/.
Choose a sequence $(x_n)\subset X$ such that $\norm{x_n}=1$ and $\norm{S_nx_n}\geq \frac{n}{n+1}$. Put $y_n = \frac{n+1}{n}S_nx_n$\/;
then by construction, $\norm{y_n}\geq 1$ for all $n$, while $\norm{Ty_n} \leq 2\norm{TS_n} \to 0$\/.
\end{proof}

\subsection*{Directly finite Banach algebras}
The notion of surjunctive pair is linked closely to the classical notion of a directly finite ring (sometimes called ``Dedekind finite'', or even just ``finite'').

\begin{defn}
A ring $R$ with identity is said to be \dt{directly finite} if each left-invertible element of $R$ is in fact invertible in~$R$\/.
\end{defn}

\begin{prop}
Let $(A,X)$ be a surjunctive pair, where $A$ is a \emph{unital} subalgebra of $\Bdd(X)$. Then $A$ is directly finite.
\end{prop}

\begin{proof}
Let $S,T\in A$ be such that $ST=I$\/. Since $\norm{S}\,\norm{Tx} \geq \norm{x}$ for all $x\in X$, we see that $T$ is injective with closed range, and hence by the surjunctivity assumption is invertible.
\end{proof}

We shall see later that the converse is false: there exist directly finite, unital Banach algebras $A$ for which the pair $(A,A)$ is not surjunctive. However, if we let $A$ act not on itself but on its dual space $A^*$\/, then we have the following result.

\begin{thm}\label{t:dual}
Let $A$ be a directly finite, unital Banach algebra. Regard $A^*$ as a left $A$-module in the natural way, and let $\imath: A \to \Bdd(A^*)$ be the corresponding (isometric, unital) embedding. Then $(\imath(A),A^*)$ is a surjunctive pair.
\end{thm}

\begin{proof}
By Lemma~\ref{l:basic}, it suffices to show that whenever $\imath(a):A^*\to A^*$ is injective with closed range, it is surjective.

In fact, we do not need the condition of having closed range; injectivity is enough. For, since we may identify $\imath(a):A^*\to A^*$ with the adjoint of the map $R_a:A\to A; x\mapsto xa$\/, when $\imath(a)$ is injective the Hahn-Banach theorem implies that $R_a$ has dense range. Let $G$ denote the group of invertible elements in $A$\/; since $G$ is open, there exists $u\in G$ and $x\in A$ such that $u=R_a(x)=xa$\/. Hence $\id[A]=u^{-1}xa$\/, and so by direct finiteness of $A$ we must have $a\in G$\/.
\end{proof}

All this is relevant to our original problem, because of the following old result.
\begin{thm}[Kaplansky]\label{t:Kap}
Let $\Gm$ be a discrete group and $\VN(\Gm)$ its group von Neumann algebra. Then $\VN(\Gm)$ is directly finite.
\end{thm}

Since direct finiteness is obviously inherited by unital subalgebras, $\ell^1(\Gm)$ is therefore directly finite. Combining Theorem~\ref{t:dual}
% the pair $(\rho_\blob(\ell^1(\Gm)),\ell^\infty(\Gm))$ -- and hence by
with Remark~\ref{rem:conjugate}, we conclude that the pair $(\ell^1(\Gm),\ell^\infty(\Gm))$ is surjunctive.
This had already been observed in previous work of the author \cite{YC_CMB}, which was another source of motivation for the present article. Indeed, the proof of Theorem~\ref{t:dual} is just an abstract version of the argument of~\cite{YC_CMB}.

\begin{rem}
Kaplansky's original proof of Theorem~\ref{t:Kap} relied on the basic theory of projections in von Neumann algebras. In~\cite{Mont} Montgomery gave a purely $\Cst$-algebraic proof that $\Cplx\Gm$ is directly finite. Inspection of her arguments shows that they extend to the larger algebra $\VN(\Gm)$, although this seems not to have been written up explicitly in the literature.
\end{rem}
\end{section}

% ridiculous kludge to get round hyperref / PDF choking on maths in bookmarks
\begin{section}{Results for directly finite \texorpdfstring{$\Cst$}{C*}-algebras}\label{s:cstar}
We start with a small observation, which is purely algebraic.
\begin{lem}\label{l:dir-fin-prod-inv}
Let $A$ be a unital, directly finite algebra, and let $x,y\in A$.
\begin{YCnum}
\item If $xy$ is invertible in $A$, then so are both $x$ and $y$.
\item $\Sp(xy)=\Sp(yx)$.
\end{YCnum}
\end{lem}

\begin{proof}
To prove (i): if there exists $h\in A$ such that $hxy=\id[A]$\/, then $y$ has a left inverse, and so -- since $A$ is directly finite -- $y$ is invertible in~$A$. Hence $x=xy\cdot y^{-1}$ is the product of two invertible elements, so is itself invertible.

In particular, if $xy$ is invertible then so is $yx$\/. It is a standard, basic result that $\Sp(xy)\setminus\{0\}=\Sp(yx)\setminus\{0\}$ in \emph{any} algebra, and thus (ii) follows.
\end{proof}

The main work needed to prove Theorem~\ref{t:headline} is contained in the following result, which to the author's knowledge is new.
\begin{thm}\label{t:on-cstar}
Let $A$ be a unital, directly finite $\Cst$-algebra.
Then $(A,A)$ is surjunctive.
\end{thm}

\begin{proof}
Let $a\in A$ and let $L_a: A\to A$ be given by $L_a(x)=ax$ for $x\in A$\/.
Suppose that $L_a$ is injective but not surjective. We shall construct an explicit sequence $(y_n)\subset A$ such that $\norm{y_n}\geq 1$ for all $n$ and $ay_n\to 0$\/.

We may suppose, \wLOG, that $\norm{a}=1$\/. Since $0\in\Sp(a)$, Lemma~\ref{l:dir-fin-prod-inv} implies that $0\in \Sp(a^*a)\subseteq [0,1]$\/.  
Let $(f_n)$ be a sequence in $C_{\Real}[0,1]$, to be determined later.
Put $y_n=f_n(a^*a)$\/. Then since $0\in \Sp(a^*a)$, we have
\[
\norm{y_n} = \norm{f_n(a^*a)}
  = \sup\{\abs{f_n(\lambda)} \;:\; \lambda\in \Sp(a^*a)\} \geq \abs{f_n(0)}\;;
\]
while, using the $\Cst$-identity and the continuous functional calculus, we obtain
\[ \begin{aligned}
\norm{ay_n}
  = \norm{(ay_n)^*ay_n}^{1/2}  
& = \norm{f_n(a^*a)a^* af_n(a^*a)}^{1/2} \\
& = \norm{a^*a f_n(a^*a)^2}^{1/2} \\
& = \sup_{\lambda\in \Sp(a^*a)} \abs{\lambda^{1/2} f_n(\lambda)} 
& \leq \sup_{0\leq \lambda\leq 1} \lambda^{1/2} \abs{f_n(\lambda)}\,.
\end{aligned} \]
The idea is now clear: take $f_n$ to satisfy $f_n(0)=1$ but to be `small outside a small neighbourhood of zero'. For sake of definiteness, put $f_n(t) = (1+nt^{1/2})^{-1}$\/: then $\norm{y_n} \geq \abs{f_n(0)} =1$ for all $n$, while
\[ \norm{ay_n} \leq \sup_{0\leq \lambda\leq 1} \frac{\lambda^{1/2}}{1+n\lambda^{1/2}} =\frac{1}{1+n} \to 0 \,.\]
This concludes the proof.
\end{proof}

\begin{proof}[Proof of Theorem~\ref{t:headline}]
The proof for the case $X=\Cst_r(\Gm)$ is immediate from Theorem~\ref{t:on-cstar}, since $\Cst_r(\Gm)$ is a directly finite $\Cst$-algebra.

It remains to treat the case of the noncommutative $L^p$-spaces. Thus, fix $1\leq p\leq\infty$, and let $X$ be the noncommutative $L^p$-space associated to $\VN(\Gm)$ (see the appendix for the definition). We now make two observations:
\begin{YCnum}
\item there is an
injective algebra homomorphism $\imath: \VN(\Gm)\to \Bdd(X)$;
\item $\imath$ has closed range, so that $\VN(\Gm)$ may be identified with a closed unital subalgebra of $\Bdd(X)$.
\end{YCnum}
These observations follow from some very basic noncommutative $L^p$-theory: as we have been unable to find a precise and concise reference for either, we have included sketch proofs and pointers to the literature in the appendix. Note however that in the case $X=\ell^2(\Gm)$ -- which, arguably, is the one of greatest interest here -- both observations are tautologically true. 

Since $\VN(\Gm)$ is a unital, directly finite $\Cst$-algebra, the pair $(\VN(\Gm),\VN(\Gm))$ is surjunctive by Theorem~\ref{t:on-cstar}. As $\imath$ has closed range, by applying Lemma~\ref{l:soi-meme} we deduce that the pair $(\imath(\VN(\Gm),X)$, and hence the pair $(\imath(\Cplx\Gm),X)$, is surjunctive. 

\end{proof}

\begin{rem}
It is not clear what can be said for the \emph{full} group $\Cst$-algebra~$\Cst(\Gm)$. In particular, it seems to be unknown if there exists a discrete group $\Gm$ such that $\Cst(\Gm)$ is \emph{not} directly finite. We content ourselves with two remarks:
\begin{YCnum}
\item If $\Gm$ is amenable, then $\Cst(\Gm)=\Cst_r(\Gm)$ is directly finite.
\item Let $\fg{2}$ denote the free group on $n$ generators. Then by
\cite[Theorem~7]{MDC_fullCF2}, there is an injective *-homomorphism
%\[ 
$\Cst(\fg{2}) \to \prod_{n=1}^\infty M_{2n}(\Cplx)$;
%\,;\]
this larger algebra is directly finite, and therefore so is $\Cst(\fg{2})$\/.
\end{YCnum}
\end{rem}

We pause for some notation and definitions. If $X$ is a Banach space, we denote by $\cK(X)$ the algebra of all compact operators on $X$\/. A~$\Cst$-algebra $A$ is said to be \dt{CCR} or \dt{liminal} if, for every irreducible $*$-representation $\pi:A\to \Bdd(H_\pi)$, we have $\pi(A)\subseteq \cK(H_\pi)$. The class of locally compact groups whose reduced $\Cst$-algebras are CCR has been intensively studied: it clearly includes all compact and all abelian groups, but also includes all connected Lie groups which are either semisimple, nilpotent, or real-algebraic . For pointers to the rather large literature, see e.g.~\cite[Chapter~7]{Fol_AHAbook}.

If $A$ is a $\Cst$-algebra, not necessarily having an identity element, let $\cu{A}$ denote its `conditional' unitization with respect to some (hence any) faithful *-representation $A\to \Bdd(H)$\/. 
%It is easily checked that if $A$ is CCR, then every irreducible, unital *-representation $\pi: \cu{A}\to \Bdd(H_\pi)$ satisfies $\pi(\cu{A})\subseteq \cu{\cK(H_\pi)}$\/.

\begin{thm}\label{t:liminal}
Let $A$ be a CCR $\Cst$-algebra. Then $(A,\cu{A})$ is a surjunctive pair.
\end{thm}

We have included this result here, since the proof goes via Theorem~\ref{t:on-cstar}. The other ingredient in the proof is the following result of Halperin~\cite{Halp_onBer}.

\begin{thm}\label{t:thinop-dirfin}
$\cu{\cK(X)}$ is directly finite for any Banach space~$X$.
\end{thm}

\begin{proof}[Proof of Theorem~\ref{t:liminal}]
Let $\id$ denote the identity of $\cu{A}$\/. By Theorem~\ref{t:on-cstar}, it suffices to prove that $\cu{A}$ is directly finite. Suppose otherwise: then there exist $a,b\in \cu{A}$ such that $ab=\id\neq ba$\/. Let $\psi$ be a pure state on $\cu{A}$ such that
\begin{equation}\label{eq:separation}
\psi((ba-\id)^*(ba-\id)) \neq 0\,,
\end{equation}
and let $\pi_\psi: \cu{A} \to \Bdd(H_\psi)$ be the corresponding GNS representation. Then $\pi_\psi$ is an irreducible *-representation of $\cu{A}$ (see~e.g.~\cite[Theorem~10.2.3]{KadRin_v2}) and it is easily checked that it is an irreducible *-representation of $A$\/; since $A$ is liminal, we have $\pi_\psi(\cu{A})\subseteq \cu{\cK(H_\psi)}$\/. We have $\pi_\psi(a)\pi_\psi(b) = \pi_\psi(\id) = I_{H_\psi}$\/, while it follows from \eqref{eq:separation} and the GNS construction that
\[ \pi_\psi(b)\pi_\psi(a)-I_{H_\psi} = \pi_\psi(ba-\id) \neq 0\,.\]
But this contradicts Theorem~\ref{t:thinop-dirfin}, and therefore $\cu{A}$ must be directly finite.
\end{proof}

Going back to Halperin's result, we can in fact prove something a little stronger, using the same kind of basic spectral theory as is implicit in~\cite{Halp_onBer}.

\begin{thm}\label{t:cpct-surjunc}
The pair $(\cK(X), X)$ is surjunctive.
\end{thm}

\begin{proof}
If $X$ is finite-dimensional, the result is trivially true (all injective endomorphisms of a finite-dimensional vector space are surjective).
Thus we may assume, \wLOG, that $X$ is infinite-dimensional.

Let $T= \lambda I + K \in \cK(X)$, where $\lambda\in\Cplx$ and $K$ is a compact operator. Suppose that $T$ is injective and non-invertible. Then $-\lambda\in\Sp(K)$; moreover, since $-\lambda$ is by assumption not an eigenvalue of $K$\/, standard spectral theory for compact operators (see e.g.~\cite[Theorem 4.25]{Rud_FA}) forces $\lambda=0$\/. Since $X$ is infinite-dimensional and $K=T$ is injective, the range of $K$ is infinite-dimensional, and so (by the open mapping theorem) cannot be closed in~$X$\/.
\end{proof}

\begin{eg}\label{eg:SOTclosure-not-surjunctive}
We can now give the example promised earlier in Remark~\ref{r:RumDoodle}.
Let $X$ be a Banach space with
the approximation property, containing a proper closed subspace isomorphic to itself. For instance, any infinite-dimensional $\ell^p$-space will suffice.
We have just seen that $(\cK(X),X)$ is surjunctive. Now, since $X$ has the approximation property, the closure of $\cK(X)$ in the strong operator topology is all of $\Bdd(X)$. Since $X$ contains a proper isomorphic copy of itself, $(\Bdd(X),X)$ is evidently not surjunctive.
\end{eg}

\begin{rem}
In Theorem~\ref{t:liminal}, we cannot relax ``CCR'' to ``GCR''.
% (a.k.a.\ ``post\-liminal'' or ``Type I''). 
For instance, the Toeplitz algebra is GCR, but (since it contains non-unitary isometries) is not even directly finite, and hence by Theorem~\ref{t:on-cstar} it cannot arise in any surjunctive pair.
\end{rem}

\end{section}

% ridiculous kludge to get round hyperref / PDF choking on maths in bookmarks
\begin{section}{The case of \texorpdfstring{$\ell^p(\Gm)$}{l-p(Gamma)} for \texorpdfstring{$p\neq 2$}{p!=2}}
Given the results obtained so far, it is tempting to wonder if, whenever $A$ is a unital and directly finite subalgebra of $\Bdd(X)$, the pair $(A,X)$ is surjunctive. The following example, which is a special case of an old construction due to Willis, shows that this is not the case.

\begin{prop}[see {\cite[Theorem 2.2]{GAW_JAuMS86b}}]
%see {\cite[proposition 10.3.1]{Dal-etal_LMSST}}]
Let $\Gm$ be a discrete group, which contains two elements $a$ and $b$ that generate a free subgroup; let $t_a$ and $t_b$ be complex scalars of modulus~$1$, and put $x= \delta_e+t_a\delta_e+t_b\delta_b\in\Cplx\Gm$\/.
Let $L_x:\ell^p(\Gm)\to\ell^p(\Gm)$ denote the left convolution operator. Then:
\begin{YCnum}
\item if $1\leq p < 2$, then $L_x$ has a left inverse in $\Bdd(\ell^p(\Gm))$, and in particular is injective with closed range;
\item if $1+t_a+t_b=0$, then $\delta_e \notin L_x(\ell^1(\Gm))$.
\end{YCnum}
\end{prop}

\begin{cor}
If $\Gm$ contains a copy of $\fg{2}$, then $(\Cplx{\Gm}, \ell^1(\Gm))$ is not a surjunctive pair.
\end{cor}

We are thus faced with the following question: for which groups $\Gamma$ is the pair $(\Cplx\Gamma, \ell^1(\Gm))$ surjunctive?
Note that by Runde's result (Theorem~\ref{t:Run_decMoore} above), all discrete Moore groups have this property; while Willis' result shows that any discrete group containing $\fg{2}$ does not have this property.
This suggests that \emph{amenability} may be the distinguishing characteristic; while we have not been able to confirm or disprove this guess, we do at least have the following theorem.

\begin{thm}\label{t:amenable}
Let $\Gm$ be a discrete amenable group,
% edited Jan 2010
let $1\leq p <\infty$, and let $a\in\ell^1(\Gm)$.
%% end edit
Let $\CV{p}(\Gm)$ denote the subalgebra of $\Bdd(\ell^p(\Gm))$ consisting of all operators that commute with right translations.

Suppose that $L_a:\ell^p(\Gm)\to\ell^p(\Gm)$ is injective and has \emph{complemented} range. Then $L_a$ is invertible in $\CV{p}(\Gm)$, and in particular is surjective. 
\end{thm}

Part of the proof uses the following standard property of modules over amenable algebras, whose proof we omit. See \cite[Theorem~2.3]{CuLo_amen} for a fairly direct argument.

\begin{lem}\label{l:reductive}
Let $A$ be an amenable Banach algebra, let $Y$ be a right Banach $A$-module, and $V$ a closed $A$-submodule of $Y$\/. Suppose that
\begin{YCnum}
\item there exists a bounded linear projection of $Y$ onto $V$\/;
\item $V$ is a dual Banach $A$-module (that is, there exists a left Banach $A$-module $X$ and a continuous $A$-module isomorphism $V\iso X^*$).
\end{YCnum}
Then there exists a bounded linear projection $P:Y\to V$ which is also a right $A$-module map.
\end{lem}

It is also convenient to use the following notation: given $c\in\ell^p(\Gm)$\/, we denote by $L_c$ the bounded operator $\ell^1(\Gm)\to \ell^p(\Gm)$ that is given by left convolution with~$c$\/.

\begin{proof}[Proof of Theorem~\ref{t:amenable}]
We start by noting that $\ell^p(\Gm)$, equipped with the usual left action of $\Gm$, is a dual Banach $\ell^1$-module.
Let $V=L_a(\ell^p(\Gm))$\/: by hypothesis this is a closed, right $\ell^1(\Gm)$-submodule of $\ell^p(\Gm)$\/, and it is a dual module, since $L_a$ defines a module isomorphism between $V$ and $\ell^p(\Gm)$\/.
Also, by hypothesis there exists a bounded linear projection of $\ell^p(\Gm)$ onto $V$\/; hence on applying Lemma~\ref{l:reductive},
there exists a bounded linear projection $P$ from $\ell^p(\Gm)$ onto the closed subspace~$V$\/, satisfying
% edited Jan 2010
\[ P(\eta*b)=P(\eta)*b \qquad\text{ for all $\eta\in \ell^p(\Gm)$ and all $b\in \ell^1(\Gm)$.} \]
% end edit
In particular, for each $\eta\in\ell^1(\Gm)$\/, we have $P(\eta)=P(\delta_e)*\eta$.
Since $P(\delta_e)\in V=L_a(\ell^p(\Gm))$\/, there exists a unique $\xi\in \ell^p(\Gm)$ such that $a*\xi=P(\delta_e)$\/. Then $a*\xi*a=P(\delta_e)*a=a$\/, and so by injectivity of $L_a$ we have $\delta_e=\xi*a$\/.

\medskip
We now consider the cases  $p=1$ and $1<p<\infty$ separately.
If $p=1$\/, then $\xi\in\ell^1(\Gm)$\/, so that $a$ is left-invertible in the algebra $\ell^1(\Gm)$\/; since this algebra is directly finite, $a$ is invertible in $\ell^1(\Gm)$ with inverse $\xi$\/.

For the remainder of this proof, we restrict attention to the case $1<p<\infty$. Here, the idea is similar to the case $p=1$, 
but since we only know at the outset that $\xi\in\ell^p(\Gm)$\/, we need to work harder (and obtain a weaker result).
The first step is to note that $L_\xi\in \CV{p}(\Gm)$\/: for, the assumption on $a$ implies there exists $\delta>0$ such that
\[ \norm{a*\psi}_p \geq \delta \norm{\psi}_p \quad\text{ for all $\psi\in\ell^p(\Gm)$,} \]
and in particular, since for every $\eta\in\ell^1(\Gm)$ we have
$a*\xi*\eta = P(\delta_e)*\eta = P(\eta)$,
we find that
\[ \norm{\xi*\eta}_p \leq \delta^{-1} \norm{a*\xi*\eta}_p \leq \delta^{-1}\norm{P(\eta)}_p \leq \delta^{-1}\norm{P}\norm{\eta}_p\,, \]
as claimed.

Since $\Gm$ is amenable, a result of Herz tells us that there exists $C_p>0$ such that
\[ \norm{T\eta}_2 \leq C_p\norm{T}_{\ell^p\to\ell^p} \norm{\eta}_2 \qquad\text{ for all $\eta\in c_{00}(\Gm)$
 and all $T\in\CV{p}(\Gm)$} \]
(see~\cite[Theorem~C]{Herz_p-space} and \cite[Theorem 5]{Herz_AIF73} for details).
Each operator in $\CV{p}(\Gm)$ can therefore be identified with a unique operator on $\ell^2(\Gm)$, which, since it commutes with right translations, will be an element of the group von Neumann algebra $\VN(\Gm)$. This gives us an injective, unital algebra homomorphism from $\CV{p}(\Gm)$ into $\VN(\Gm)$, and since the latter is directly finite we see that $\CV{p}(\Gm)$ is also directly finite. Since $L_a L_\xi=I$ in $\CV{p}(\Gm)$, we conclude that $L_a$ is invertible in $\CV{p}(\Gm)$ with inverse $L_\xi$\/.
\end{proof}

%\begin{rem}\label{rem:cavil}
%The result of Herz that we used in this proof splits naturally into two parts. The first part shows that a certain algebra $\operatorname{PM}_p(\Gm)$\/, which may be defined as the closure of $\ell^1(\Gm)$ in the ultraweak topology of $\Bdd(\ell^p(\Gm))$, embeds naturally into $\operatorname{PM}_2(\Gm)=\CV(\Gm)=\VN(\Gm)$\/; this part is \cite[Theorem C]{Herz_p-space}, and uses amenability (in the form of Reiter's $P_1$-condition, ultimately. The second part is to show that, when $\Gm$ is amenable, the natural, isometric inclusion $\operatorname{PM}_p(\Gm)\hookrightarrow \CV{p}(\Gm)$ is actually surjective; this is \cite[Theorem 5]{Herz_AIF73}, and also uses amenability.
%\end{rem}
\end{section}

\begin{section}{Closing questions and further work}
While the preceding results give complete answers for the regular representation of a discrete group on its $\ell^2$-space, we have seen that the problem of determining surjunctivity of $(\Cplx\Gm,\ell^p(\Gm))$ remains unresolved for $p\neq 2$ and non-Moore groups. We therefore close with a list of questions which this work has raised, and which we hope may stimulate further investigation.

\begin{qn}
Let $1<p<\infty$, $p\neq 2$\/.
Is $(\Cplx\fg{2},\ell^p(\fg{2}))$ surjunctive?
\end{qn}

\begin{qn}\label{q:Tessera}
Let $H_3(\Zahl)$ denote the $3$-dimensional, integer Heisenberg group. Is
% the pair
%\hfill\break
$(\Cplx H_3(\Zahl), \ell^1(H_3(\Zahl)))$ surjunctive? What about $\ell^p(H_3(\Zahl))$ for $1<p<\infty$, $p\neq 2$\/?
\end{qn}

There are several reasons for considering $H_3(\Zahl)$\/. Firstly, as a two-step nilpotent group it is perhaps the next step into noncommutativity after Moore groups (by results of Thoma, every discrete Moore group has an abelian normal subgroup of finite index, see~\cite[Theorem~7.8]{Fol_AHAbook}). Secondly, in view of the role played by spectral arguments and functional calculus in Section~\ref{s:cstar}, it may be relevant that  $\ell^1(H_3(\Zahl))$ admits a $C^k$-functional calculus for some $k<\infty$\/, by results of Dixmier; however, it is not clear if this can be used to construct approximate eigenvectors, as done for the case of the $\Cst$-norm.
% One technical problem is that it is not clear how to relate the norms of $f(a^*a)$ and $f(aa^*)$ when $a$ is an arbitrary element of $\ell^1(H_3(\Zahl)$ and $f$ is holomorphic on some neighbourhood of $\Sp(aa^*)=\Sp(a^*a)$\/.

\begin{qn}
The result of Runde that was mentioned in the introduction implies, as a very special case, that $(\ell^1(\Gm),\ell^1(\Gm))$ is surjunctive whenever $\Gm$ is a discrete Moore group. Is there a more direct proof of this?
\end{qn}

For instance, if $\Gm$ is abelian then surjunctivity can be proved easily by considering the Gelfand transform of $\ell^1(\Gm)$ and noting that all its maximal ideals have bounded approximate identities. It would be interesting to see an analogous argument for Moore groups, which did not require the full machinery deployed in Runde's article.
\subsection*{Acknowledgements}
The author thanks Matthew Daws for comments on an early version of this article, and for helpful conversations regarding the articles~\cite{Herz_p-space,Herz_AIF73}.
\end{section}

\appendix

% ridiculous kludge to get round hyperref / PDF choking on maths in bookmarks
\begin{section}{The action of \texorpdfstring{$\VN(\Gm)$}{VN(Gamma)} on its noncommutative \texorpdfstring{$L^p$}{L-p}-spaces}
In this short appendix we collect, for convenience, the definition and some basic properties of the noncommutative $L^p$-spaces associated to a group von Neumann algebra. Rather than refer the reader to sources that treat such constructions in full generality and sophistication, we shall present the background needed to justify the statements in the proof of Theorem~\ref{t:headline}.
For more details see \cite[\S\S2--3]{Dix_BSMF53} or \cite[\S1]{PiXu_handbook}.

Throughout this section $\cM$ is a fixed von Neumann algebra, equipped with a faithful, normal, \emph{finite} tracial state $\tau$. (In the case $\cM=\VN(\Gm)$, the trace $\tau$ is the canonical one given by $T\mapsto \pair{T\delta_e}{\delta_e}$.)

\begin{defn}
Let $x\in\cM$\/. For $1\leq p <\infty$ and $x\in\cM$, put
\begin{equation}
\pnorm{p}{x} \defeq \tau((x^*x)^{p/2})^{1/p}
\end{equation}
and for $p=\infty$ put $\pnorm{\infty}{x} = \norm{x}$.
Then $\pnorm{p}{\cdot}$ defines a norm on $\cM$, and we denote the completion of $(\cM,\pnorm{p}{\cdot})$ by $L^p(\cM,\tau)$.
\end{defn}

\begin{rem}
It is worth identifying $L^p(\cM,\tau)$ in the cases $p=1,2$ and~$\infty$.
By definition, $L^\infty(\cM,\tau)\equiv \cM$. 
It is also clear from the definition that $L^2(\cM,\tau)$ is nothing but the Hilbert space associated to $\cM$ by the GNS construction for the state $\tau$\/.
Finally, it turns out that elements of $L^1(\cM,\tau)$ correspond to ultraweakly continuous functionals on $\cM$\/, and hence $L^1(\cM,\tau)$ is naturally identified with the isometric predual $\cM_*$.  (See e.g.~\cite[Section 2, Th\'eor\`eme 5]{Dix_BSMF53}.)

When $\cM=\VN(\Gm)$ and $\tau$ is the canonical trace $T\mapsto \pair{T\delta_e}{\delta_e}$, we can therefore identify $L^2(\VN(\Gm),\tau)$ with $\ell^2(\Gm)$ and $L^1(\VN(\Gm),\tau)$ with the Fourier algebra $A(\Gm)$.
\end{rem}

\begin{lem}
Let $1\leq p\leq\infty$. Then
$\pnorm{p}{ax} \leq \pnorm{\infty}{a} \pnorm{p}{x}$
for all $a,x\in\cM$\/.
Consequently, the left action of $\cM$ on itself extends to give a continuous left action of $\cM$ on $L^p(\cM,\tau)$, and hence a continuous algebra homomorphism $\imath_p: \cM\to \Bdd(L^p(\cM,\tau))$.
\end{lem}
We refer to \cite{Dix_BSMF53} for the proof. (Briefly: for $p=\infty$ the claim is tautologous, and for $p=1$ it is \cite[Th\'eor\`eme 4]{Dix_BSMF53}. The general case follows from \cite[p.~26, Corollaire~3]{Dix_BSMF53}.)
% which is itself obtained via Riesz-Thorin interpolation.)

\medskip
The following lemma is surely well known to specialists, but the author does not know of any explicit reference for it.
\begin{lem}
Let $a\in\cM$ and let $1\leq p \leq \infty$. Let $\cC$ be the closed unital subalgebra of $\cM$ generated by $a^*a$. Then
\[ \norm{a} \geq \sup\{ \pnorm{p}{ax} \;:\; x\in \cC, \pnorm{p}{x} \leq 1\} \,.\]
In particular, the embedding $\imath_p:\cM\to\Bdd(L^p(\cM,\tau))$ has closed range.
\end{lem}

As the proof needs only basic $\Cst$-algebraic tools, we give a quick sketch. 
Without loss of generality, assume that $1\leq p <\infty$ and that $\norm{a}=\norm{a^*a}=1$. Identify $\cC$ with $C(\Sp(a^*a))$ and let $\mu$ be the probability measure on $\Sp(a^*a)$ that corresponds to $\tau\vert_{\cC}$\/. Let $\veps>0$\/; then by duality and basic measure theory, there exists a continuous positive function $f:\Sp(a^*a)\to[0,\infty)$ such that
\[ \int_{\Sp(a^*a)}f(t)\,d\mu(t)=1 \quad\text{ and }\quad \veps+ \int_{\Sp(a^*a)} t^{p/2} f(t)\,d\mu(t) \geq \sup_{t\in \Sp(a^*a)} \abs{t^{p/2}} = 1\,.\]
Put $x=f(a^*a)^{1/p} =f^{1/p}(a^*a)$\/, and observe that
\[ \pnorm{p}{x}^p = \tau((x^*x)^{p/2}) = \tau( f(a^*a)) = 1 \,,\]
while, since $x=x^*$ commutes with $a^*a$,
\[ \pnorm{p}{ax}^p = \tau( (x^*a^*ax)^{p/2}) = \tau((a^*a x^*x)^{p/2}) = \tau((a^*a)^{p/2} f(a^*a)) \geq 1- \veps \,.\]
Since $\veps$ was arbitrary, the result follows.
\end{section}

%\bibliography{temporary}
%\bibliographystyle{siam}
%\input{bib-by-hand}

\vfill

\noindent%
\begin{tabular}{l}
Y. Choi\\
D\'epartement de math\'ematiques
% \\
% \text{\hspace{1.0em}} 
et de statistique,\\
Pavillon Alexandre-Vachon\\
Universit\'e Laval\\
Qu\'ebec, QC \\
Canada, G1V 0A6 \\
{\bf Email: \tt y.choi.97@cantab.net}
\end{tabular}
\end{document}